\documentclass[11pt]{article}
\usepackage{amsmath}
\usepackage{dsfont}
\usepackage{mathrsfs}
\usepackage{amsmath,amssymb}
\usepackage{amsfonts}
\usepackage{hyperref}
\usepackage{amsthm}
\usepackage{graphicx}
\usepackage{subfigure}
\usepackage{xcolor}
\renewcommand{\qed}{\hfill\small{$\square$}\normalsize}

\hfuzz=\maxdimen
\theoremstyle{definition}
\newtheorem{lemma}{Lemma}[section]
\newtheorem{definition}[lemma]{Definition}
\newtheorem{proposition}[lemma]{Proposition}
\newtheorem{theorem}[lemma]{Theorem}
\newtheorem{corollary}[lemma]{Corollary}
\newtheorem{example}{Example}
\newtheorem{remark}{Remark}

\newtheorem{conjecture}{Conjecture}
\numberwithin{equation}{section}

\renewcommand{\qed}{\hfill\small{$\square$}\normalsize}

\DeclareFixedFont{\Acknowledgment}{OT1}{cmr}{bx}{n}{14pt}
\begin{document}

\title{\bf $\alpha$-curvatures and $\alpha$-flows on low dimensional triangulated manifolds}
\author{Huabin Ge, Xu Xu}
\maketitle

\begin{abstract}
In this paper, we introduce two discrete curvature flows, which are called $\alpha$-flows on two and three dimensional triangulated manifolds. For triangulated surface $M$, we introduce a new normalization of combinatorial Ricci flow (first introduced by Bennett Chow and Feng Luo \cite{CL1}),
aiming at evolving $\alpha$ order discrete Gauss curvature to a constant.
When $\alpha\chi(M)\leq0$, we prove that the convergence of the flow is equivalent to the existence of constant $\alpha$-curvature metric. We further get a necessary and sufficient combinatorial-topological-metric condition, which is a generalization of Thurston's combinatorial-topological condition, for the existence of constant $\alpha$-curvature metric. For triangulated 3-manifolds, we generalize the combinatorial Yamabe functional and combinatorial Yamabe problem introduced by the authors in \cite{GX2,GX4} to $\alpha$-order. We also study the $\alpha$-order flow carefully,
aiming at evolving $\alpha$ order combinatorial scalar curvature, which is a generalization of Cooper and Rivin's combinatorial scalar curvature, to a constant.
\end{abstract}

\section{Introduction}\label{Introduction}
Calculations in local coordinate charts are indispensable for the study of smooth manifolds. However, an entirely different way came from Regge \cite{Re}. The basic procedure is to triangulate a manifold to simplexes, and finally obtain the geometrical and topological information of the manifold by studying the geometry in a simplicial complex and the combinatorial structure in the triangulation. On triangulated manifolds, the most general discrete metric seems to be piecewise linear metric which is defined on all edges. Discrete curvatures are determined by discrete metrics. Roughly speaking, discrete metrics are edge lengthes while discrete curvatures are angles between sub-simplexes and their combinatorics. Besides the general piecewise linear metrics, there are discrete metrics defined on all vertices. This type of metrics may be considered as a discrete conformal class. In his work on constructing hyperbolic metrics on 3-manifolds, Thurston \cite{T1} introduced circle packing metric on a triangulated surface with prescribed intersection angles. For triangulated 3-manifolds, Cooper and Rivin introduced a sphere packing metric. We shall study circle (sphere) packing metrics and the corresponding discrete curvatures in the following.

Suppose $M$ is a closed surface with triangulation $\mathcal{T}=\{V,E,F\}$, where $V,E,F$ represent the sets of vertices, edges and faces respectively.
 $\Phi: E\rightarrow [0,\frac{\pi}{2}]$ is a function evaluating each edge $i\sim j$ a weight $\Phi_{ij}$.
 The triple $(M, \mathcal{T}, \Phi)$ will be referred as a weighted triangulation of $M$ in the following. All the vertices are supposed to be ordered one by one, marked by $1, \cdots, N$, where $N=|V|$ is the number of vertices. Throughout this paper, all functions $f: V\rightarrow \mathds{R}$ will be regarded as column vectors in $\mathds{R}^N$ and $f_i$ is the value of $f$ at $i$. And we use $C(V)$ to denote the sets of functions defined in this way.

Each map $r:V\rightarrow (0,+\infty)$ is called a circle packing metric. Given $(M, \mathcal{T}, \Phi)$, we can attach each edge $\{ij\}$ a length $l_{ij}=\sqrt{r_i^2+r_j^2+2r_ir_j\cos \Phi_{ij}}$. It is proved \cite{T1} that the lengths $\{l_{ij}, l_{jk}, l_{ik}\}$ satisfies the triangle inequality for each face $\{ijk\}\in F$, which ensures that the face $\{ijk\}$ could be realized as a Euclidean triangle with lengths $\{l_{ij}, l_{jk}, l_{ik}\}$. Suppose $\theta_i^{jk}$ is the inner angle of triangle $\{ijk\}$
at the vertex $i$, the classical discrete Gauss curvature at the vertex $i$ is defined as
\begin{equation}\label{Def-Gauss curv}
K_i=2\pi-\sum_{\{ijk\}\in F}\theta_i^{jk},
\end{equation}
where the sum is taken over all the triangles with $i$ as one of its vertices. Then
\begin{equation*}\label{Gauss-Bonnet without weight}
\sum_{i=1}^NK_i=2\pi \chi(M),
\end{equation*}
which may be considered as a discrete version of Gauss-Bonnet formula. Denote $K_{av}=2\pi \chi(M)/N$. Notice that constant $K$-curvature metric, i.e. a metric $r$ with $K_i=K_{av}$ for all $i\in V$, does not always exist. In fact, for any proper subset $I\subset V$, let $F_I$ be the subcomplex whose vertices are in $I$ and let $Lk(I)$ be the set of pairs $(e, v)$ of an edge $e$ and a vertex $v$ such the following three conditioins: (1) the end points of $e$ are not in $I$; (2) $v$ is in $I$; (3) $e$ and $v$ form a triangle. Thurston \cite{T1} proved that the existence of constant $K$-curvature metric is equivalent the following combinatorial-topological conditions
\begin{equation}\label{condition-Thurston}
2\pi\chi(M)\frac{|I|}{|V|} >-\sum_{(e,v)\in Lk(I)}(\pi-\Phi(e))+2\pi\chi(F_I), \;\;\forall I: \phi\subsetneqq I\subsetneqq V.
\end{equation}
Moreover, the constant $K$-curvature metric is unique, if exists, up to the scaling of $r$.

Chow and Luo \cite{CL1} first established an intrinsic connection between Thurston's circle packing and surface Ricci flow.
They introduced a combinatorial Ricci flow
\begin{equation}\label{Def-ChowLuo's flow}
\frac{dr_i}{dt}=-K_ir_i
\end{equation}
and its normalization
\begin{equation}\label{Def-ChowLuo's normalized flow}
\frac{dr_i}{dt}=(K_{av}-K_i)r_i.
\end{equation}
Then they proved that flow (\ref{Def-ChowLuo's normalized flow}) converges iff. the constant $K$-curvature metric exists, iff. Thurston's combinatorial and topological conditions (\ref{condition-Thurston}) are satisfied.

Inspired by Chow and Luo's work, the first author introduced a combinatorial Calabi flow
\begin{equation}
\frac{dr_i}{dt}=\Delta K_ir_i
\end{equation}
in \cite{Ge} and proved similar results. The authors also studied the corresponding problem in hyperbolic background
geometry using combinatorial Calabi flow \cite{GX1}.

The paper is organized as follows. In Section 2, we introduce the $\alpha$th order Yamabe flow for surfaces and use it to study the
corresponding constant and prescribing curvature problem. In Section 3, we study the constant $\alpha$-curvature problem for
triangulated 3-manifolds. And Section 4 is devoted to some useful lemmas used in the paper.

\section{$\alpha$ order combinatorial flow in two dimension}
\label{section-2d}
As pointed by the authors in \cite{GX4}, there are two disadvantages of the classical definition of $K_i$. For one thing, classical discrete Gauss curvature does not perform so perfectly in that it is scaling invariant, i.e. if $\tilde{r}_i=\lambda r_i$ for some positive constant $\lambda$, then $\tilde{K}_i=K_i$, which is different from the transformation of scalar curvature $R_{\lambda g}=\lambda ^{-1}R_g$ in smooth case. For another, classical discrete Gauss curvature can't be used directly to approximate smooth Gauss curvature since it always tends to zero as the triangulation becomes finer and finer. To amend this flaw, we suggested a new definition of discrete Gauss curvature by dividing an ``area element" $A_i$.

We observed that the easiest form of ``area element" may be $A_i=\pi r_i^{\alpha}$, where $\alpha$ is any real number. For convenience, we omitted the coefficient $\pi$ in the following definition.
\begin{definition}\label{def-R-curvature}
Given a weighted triangulated surface $(M, \mathcal{T}, \Phi)$ with circle packing metric $r: V\rightarrow (0,+\infty)$,
the discrete Gauss curvature of order $\alpha$ (``$\alpha$-curvature" for short) at the vertex $i$ is defined to be
\begin{equation}
R_{\alpha,i}=\frac{K_i}{r_i^{\alpha}},
\end{equation}
where $K_i$ is the classical discrete Gauss curvature defined as the angle deficit at $i$ by (\ref{Def-Gauss curv}).
\end{definition}

For the classical discrete Gauss curvature $K$, \cite{T1,MR,CL1} provide a complete description of
the space of admissible $K$-curvature $\{K=K(r)|r\in \mathds{R}^N_{>0}\}$, see also Lemma \ref{Lemma-admissible-curvature space}
in this paper. Similarly, it's interesting to know how to describe the space of admissible $\alpha$-curvatures.

\begin{remark}
For the special case $\alpha=2$, $A_i=\pi r_i^2$ is just the area of the disk packed at $i$. This case is especially interesting.
It was shown \cite{GX4} that $r_i^2$ is a suitable analogue of smooth Riemannian metric tensor $g$.
Furthermore, the $2$-curvature is very similar to the smooth Gauss curvature.
On one hand, it has similar scaling law; On the other hand, it can be used to approximate the smooth Gauss curvature.
\end{remark}

One may ask if there are circle packing metrics whose corresponding $\alpha$-curvatures are constants or have prescribing curvatures.
Curvature flow method gives an efficient approach to study these questions. The smooth 2-dimensional Yamabe flow, which is the same as Ricci flow, is $\partial g/\partial t=(r-R)g$. It is Chow and Luo's insight to consider $r_i$ as an analogue of smooth Riemannian metric tensor $g$, $K_i$ as an analogue of smooth Riemannian curvature $K=R/2$ and $K_{av}$ as an analogue of smooth average curvature $r$. Thus the classical combinatorial Ricci flow (\ref{Def-ChowLuo's normalized flow}) comes into being. Inspired by their work and the idea of taking $r_i^2$ as a suitable analogue of smooth Riemannian metric tensor $g$ and $K_i/r_i^2$ as an appropriate discrete version of Gaussian curvature $K$, we \cite{GX4} defined a combinatorial Ricci flow as
\begin{equation}\label{Def-new alpha-Ricci flow}
\frac{dg_i}{dt}=(R_{av}-R_i)g_i,
\end{equation}
where $g_i=r_i^2$ and $R_{av}=2\pi\chi(M)/\|r\|^2$.
We further proved that, for surfaces with non-positive Euler characteristic, the flow (\ref{Def-new Ricci flow}) converges iff.
there exists a constant $2$-curvature metric.
For $\alpha$-curvature, we proved similar results for surface $M$ with $\alpha \chi(M)\leq0$ by introducing a modified flow
\begin{equation}\label{Def-new Ricci flow}
\frac{dr_i}{dt}=(R_{\alpha,av}-R_{\alpha,i})r_i,
\end{equation}
where $R_{\alpha,av}=2\pi\chi(M)/\|r\|^{\alpha}_{\alpha}$. However, there are no similar results for surface $M$ with $\alpha\chi(M)>0$. Even worse, we don't know whether the solution of flow (\ref{Def-new alpha-Ricci flow}) exists for all time $t\in(-\infty,+\infty)$.
For this reason, we introduce another flow here, the solution of which always exists for all time $t\in \mathds{R}$.

Set $u_i=\ln r_i$.
\begin{definition}
Given $(M, \mathcal{T}, \Phi)$ with circle packing metric $r$, denote $s_{\alpha}=2\pi\chi(M)/\|r\|^{\alpha}_{\alpha}$. The $\alpha$th order combinatorial Ricci (Yamabe) flow (``$\alpha$-flow" for short) is
\begin{equation}\label{Def-Yamabe flow-2d}
\frac{du_i}{dt}=s_{\alpha}r_i^{\alpha}-K_i.
\end{equation}
\end{definition}
\begin{remark}
Chow and Luo's normalized Ricci flow (\ref{Def-ChowLuo's normalized flow}) is in fact the $0$-flow defined above with $\alpha=0$. Hence $\alpha$-flow (\ref{Def-Yamabe flow-2d}) may be considered as a different normalization of Chow and Luo's flow (\ref{Def-ChowLuo's flow}).
\end{remark}

Formally, $\alpha$-flow seems plainer and simpler than flow (\ref{Def-new Ricci flow}). Furthermore, we have
\begin{proposition}\label{Prop-longtime-existence}
The solution to discrete Yamabe flow (\ref{Def-Yamabe flow-2d}) exists for all time $t\in \mathds{R}$.
\end{proposition}
\begin{proof}
Notice that $|s_{\alpha}r_i^{\alpha}|\leq2\pi|\chi(M)|$, $(2-d)\pi\leq K_i\leq 2\pi$, where $d=\max\limits_{1\leq i\leq N} d_i$ and $d_i$ is the degree of vertex $i$. Hence $|s_{\alpha}r_i^{\alpha}-K_i|$ is uniformly bounded by a constant depending only on the topology and the triangulation. Then by the extension theorem for solution in ODE theory, the solution to (\ref{Def-Yamabe flow-2d}) exists for all $t\in(-\infty,+\infty)$.\qed\\
\end{proof}

\begin{proposition}\label{Prop-negative-gradient-flow}
$\alpha$-flow (\ref{Def-Yamabe flow-2d}) is a negative gradient flow. Moreover, $\prod_{i=1}^N r_i(t)$ preserves to be a constant along $\alpha$-flow.
\end{proposition}
\begin{proof}
It is remarkable that Chow and Luo \cite{CL1} introduced a functional
\begin{equation*}
F(u)=\int_{u_0}^u\sum\nolimits_{i=1}^N\left(K_i-K_{av}\right)du_i,
\end{equation*}
which is called discrete Ricci potential, to prove the convergence of their Ricci flow (\ref{Def-ChowLuo's normalized flow}). With a slight modification, we define the $\alpha$-order discrete Ricci potential (``$\alpha$-potential" for short) as
\begin{equation}%\label{Def-alpha-Ricci potential}
F(u)=\int_{u_0}^u\sum\nolimits_{i=1}^N\left(K_i-s_{\alpha}r_i^{\alpha}\right)du_i,
\end{equation}
where $u_0\in\mathds{R}^N$ is arbitrary selected. The $\alpha$-potential is well defined since $\sum(K_i-s_{\alpha}r_i^{\alpha})du_i$ is a closed differential form. Thus $\alpha$-flow is in fact $\dot{u}=-\nabla_u F$.

Furthermore, $(\sum u_i)'=\sum (s_{\alpha}r_i^{\alpha}-K_i)=0$ implies that $\sum u_i(t)$ and hence $\prod r_i(t)$ are invariant along the $\alpha$-flow.\qed\\
\end{proof}

For this reason, we always assume the initial circle packing metric $r(0)$ belongs to the hypersurface
 $\prod_{i=1}^N r_i=1$ in the following.

\begin{proposition}\label{Prop-converge imply CCCP-metric}
If the solution to $\alpha$-flow (\ref{Def-Yamabe flow-2d}) converges, then there exists at least a constant $\alpha$-curvature metric.
\end{proposition}
\begin{proof}
Notice that constant $\alpha$-curvature metrics are exactly the critical points of ODE system (\ref{Def-Yamabe flow-2d}). If the flow converges, it
 converges to its critical points. This implies the existence of constant $\alpha$-curvature metric. \qed\\
\end{proof}

Discrete Laplace operators are closely related to discrete curvature flows. Denote
\begin{equation}
L=\frac{\partial(K_1, \cdots, K_N)}{\partial (u_1,\cdots,u_N)}
\end{equation}
as the Jacobian matrix of the curvature map. It's very interesting that $L: C(V)\rightarrow C(V)$ with $f\mapsto Lf$ can be considered as the prototype of a type of discrete Laplace operator \cite{CL1,G3,Ge,GX4}, which comes from the dual structure of circle patterns. Chow and Luo proved that $L$ is positive semi-definite everywhere in $\mathds{R}_{>0}^N$. Moreover, $rank(L)=N-1$ and the kernel of $L$ is $t(1,\cdots,1)^T$, $t\in \mathds{R}$. Other properties of $L$ is exhibited in Lemma \ref{Lemma-property of L}. The authors once defined \cite{GX4} a two dimensional $\alpha$-order combinatorial Laplacian as
\begin{equation}\label{Def-alpha Laplacian-2d}
\Delta_\alpha f_i=-\frac{1}{r_i^\alpha}\sum_{j=1}^N \frac{\partial K_i}{\partial u_j} f_j=-\frac{1}{r_i^\alpha}\sum_{j\thicksim i} \frac{\partial K_i}{\partial u_j}(f_j-f_i).
\end{equation}
Recall that we have considered each $f\in C(V)$ as a column vector, hence the two dimensional $\alpha$-Laplacian (\ref{Def-alpha Laplacian-2d}) can be written in a matrix form,
\begin{equation}
\Delta_{\alpha}=-\Sigma^{-\alpha}L
\end{equation}
with $\Delta_{\alpha} f=-\Sigma^{-\alpha}Lf$ for each $f\in C(V)$, where $\Sigma=diag\big\{r_1,\cdots,r_N\big\}$. It's interesting that we can consider $r_i^\alpha$ as an analogy of $d\mu$, where $d\mu$ is the area element in smooth case. We can define a $\alpha$ order inner product $\langle\cdot, \cdot\rangle_{\alpha}$ on $(M, \mathcal{T}, \Phi)$ with circle packing metric $r$ by
\begin{equation}\label{inner product}
\langle f, h \rangle_{\alpha}=\sum_{i=1}^{N}f_ih_ir_i^{\alpha}=h^T\Sigma^{\alpha}f
\end{equation}
for real functions $f, h\in C(V)$. Then the $\alpha$-Laplacian $\Delta_{\alpha}: C(V)\rightarrow C(V)$ is self-adjoint since
$$\langle \Delta_{\alpha}f, h\rangle_{\alpha}=\langle f, \Delta_{\alpha}h\rangle_{\alpha}$$
for any $f, h\in C(V)$.

Denote the first positive eigenvalue of $-\Delta_{\alpha}$ as $\lambda_1(-\Delta_{\alpha})$.
The following theorem shows that the first positive eigenvalue of
$\alpha$ order combinatorial Laplace operator is closely related to the behavior of discrete $\alpha$-flow.

\begin{theorem}\label{Thm-convergence-nonpositive Euler number}
Assuming $\lambda_1(-\Delta_{\alpha})>\alpha s_{\alpha}=\alpha\frac{2\pi\chi(M)}{\|r\|_{\alpha}^{\alpha}}$ at all $r\in \mathds{R}_{>0}^N$. Then the $\alpha$-flow (\ref{Def-Yamabe flow-2d}) converges iff. there exists at least a constant $\alpha$-curvature metric on $(M, \mathcal{T}, \Phi)$.
\end{theorem}
\begin{proof}
We give a direct and self-contained proof here. We just need to prove the ``if" part. Set $\Lambda_{\alpha}=\Sigma^{-\frac{\alpha}{2}}L\Sigma^{-\frac{\alpha}{2}}$, then
$$\Delta_{\alpha}=-\Sigma^{-\frac{\alpha}{2}}\Lambda_{\alpha}\Sigma^{\frac{\alpha}{2}},$$
which implies that $\lambda_1(-\Delta_{\alpha})=\lambda_1(\Lambda_{\alpha})$. Assuming $r^*$ is a constant $\alpha$-curvature metric.
Scaling $r^*$ to any $tr^*$ with $t>0$, the corresponding $\alpha$-curvature is still a constant. Hence we may suppose $r^*$ belongs to the hypersurface $\prod r_i=1$. Denote $u^*$ as the $u$-coordinate of $r^*$, and $\alpha$-potential as
\begin{equation}\label{Def-alpha-potential}
F(u)=\int_{u^*}^u\sum_{i=1}^N\left(K_i-s_{\alpha}r_i^{\alpha}\right)du_i.
\end{equation}
Denote $r^{\alpha}=(r_1^{\alpha},\cdots,r_N^{\alpha})^T$, then it's easy to calculate
\begin{equation}\label{Hession of F}
\begin{aligned}
Hess_uF
=L-\alpha s_{\alpha}
\left(\Sigma^{\alpha}-\frac{r^{\alpha}(r^{\alpha})^T}{\|r\|_{\alpha}^{\alpha}}\right)=\Sigma^{\frac{\alpha}{2}}
\left(\Lambda_{\alpha}-\alpha s_{\alpha}\left(I-\frac{r^{\frac{\alpha}{2}}(r^{\frac{\alpha}{2}})^T}{\|r\|^{\alpha}_{\alpha}}\right)\right)
\Sigma^{\frac{\alpha}{2}}.
\end{aligned}
\end{equation}
Choose an orthonormal matrix $P$ such that $P^T\Lambda_{\alpha} P=diag\{0,\lambda_1(\Lambda_{\alpha}),\cdots,\lambda_{N-1}(\Lambda_{\alpha})\}$.
Suppose $P=(e_0,e_1,\cdots,e_{N-1})$,
where $e_i$ is the $(i+1)$-column of $P$.
Then $\Lambda_{\alpha} e_0=0$ and $\Lambda_{\alpha} e_i=\lambda_i e_i,\,1\leq i\leq N-1$,
which implies $e_0=r^\frac{\alpha}{2}/\|r^\frac{\alpha}{2}\|$ and $r^\frac{\alpha}{2}\perp e_i,\,1\leq i\leq N-1$.
Hence $\big(I-\frac{r^{\frac{\alpha}{2}}(r^{\frac{\alpha}{2}})^T}{\|r\|^{\alpha}_{\alpha}}\big)e_0=0$ and $\big(I-\frac{r^{\frac{\alpha}{2}}(r^{\frac{\alpha}{2}})^T}{\|r\|^{\alpha}_{\alpha}}\big)e_i=e_i$, $1\leq i\leq N-1$,
which implies
%$$P^T\big(I-\frac{r^{\frac{\alpha}{2}}(r^{\frac{\alpha}{2}})^T}{\|r\|^{\alpha}_{\alpha}}\big)P=diag\{0,1,\cdots,1\}.$$ and hence
$$Hess_uF=\Sigma^{\frac{\alpha}{2}}P diag\big\{0,\lambda_1(\Lambda_{\alpha})-\alpha s_{\alpha},\cdots,\lambda_{N-1}(\Lambda_{\alpha})-\alpha s_{\alpha}\big\}\Sigma^{\frac{\alpha}{2}}P^T.$$
If $\lambda_1(\Lambda_{\alpha})>\alpha s_{\alpha}=\alpha\frac{2\pi\chi(M)}{\|r\|_{\alpha}^{\alpha}}$, then $HessF\geq0$, $rank(F)=N-1$. Using Lemma \ref{Lemma-positive-definite}, Lemma \ref{Lemma-proper} and Lemma \ref{Lemma-injective}, we know that $F$ is proper on $\mathscr{U}$, where $\mathscr{U}\triangleq\{u\in \mathds{R}^N|\sum_i u_i=0\}$. Moreover, $\lim\limits_{u\in\mathscr{U},\,u\rightarrow \infty}F(u)=+\infty$ and $u^*$ is the unique zero point of $\nabla F$ which is also the unique minimum point of $F$ on $\mathscr{U}$. Let $\varphi(t)=F(u(t))$, then $\varphi'(t)=-\sum_i\left(K_i-s_{\alpha}r_i^{\alpha}\right)^2=-\|\nabla F\|^2\leq0$, which implies that $u(t)$ lies in a compact region of $\mathscr{U}$. We can further get $\varphi''(t)=2(K-s_{\alpha}r^{\alpha})^THessF(K-s_{\alpha}r^{\alpha})\geq0$. The fact that $\varphi'\leq0$, $\varphi''\geq0$ and $\varphi$ is bounded below implies $\varphi'(+\infty)=0$. Combining $u^*$ is the unique zero point of $\nabla F$ and $\{u(t)\}\subset\subset\mathscr{U}$, we know $u(t)\rightarrow u^*$ exponentially fast.
\qed\\
\end{proof}

If $\alpha\chi(M)\leq 0$, we always have $\lambda_1(-\Delta_{\alpha})>0\geq \alpha s_{\alpha}=\alpha\frac{2\pi\chi(M)}{||r||^{\alpha}_{\alpha}}$, thus we have
\begin{theorem}\label{Thm-convergence-negative Euler}
Suppose $(M, \mathcal{T}, \Phi)$ is a weighted triangulated surface with $\alpha\chi(M)\leq 0$.
Then the solution to the $\alpha$-flow (\ref{Def-Yamabe flow-2d}) converges if and only if there exists a constant $\alpha$-curvature metric $r^*$. Furthermore, if the solution converges, it converges exponentially fast to the metric of constant curvature.
\end{theorem}

\begin{remark}
When $\alpha=0$, above theorem is obtained in \cite{CL1}.
\end{remark}

Next we want to study the combinatorial and topological conditions for the existence of constant $\alpha$-curvature metric. For any proper subset $I\subset V$, denote
$$\mathscr{Y}_I\triangleq\{x\in \mathds{R}^N |\sum_{i\in I}x_i >-\sum_{(e,v)\in Lk(I)}(\pi-\Phi(e))+2\pi\chi(F_I)\}$$
and $\mathscr{K}_{GB}\triangleq\{x\in \mathds{R}^N|\sum_{i=1}^Nx_i=2\pi\chi(X)\}$. The authors once proved
\begin{proposition}\label{Proposition-topo-com-condition-ofGeXu}
(\cite{GX4}, Theorem 2.34)
Given a weighted triangulated surface $(M, \mathcal{T}, \Phi)$. Consider Thurston's circle packing metric and the $\alpha$-curvature. When $\alpha>0$ and $\chi(M)<0$, the existence of constant $\alpha$-curvature metric is equivalent to $\mathscr{K}_{GB} \cap (\mathop{\cap}_{\phi\neq I \subsetneqq V} \mathscr{Y}_I )\cap\mathds{R}^N_{<0}\neq\phi$; When $\alpha<0$ and $\chi(M)>0$, the existence of constant $\alpha$-curvature metric is equivalent to $\mathscr{K}_{GB} \cap (\mathop{\cap}_{\phi\neq I \subsetneqq V} \mathscr{Y}_I )\cap\mathds{R}^N_{>0}\neq\phi$; When $\chi(M)=0$, the existence of constant $\alpha$-curvature metric is equivalent to $\mathscr{K}_{GB} \cap (\mathop{\cap}_{\phi\neq I \subsetneqq V} \mathscr{Y}_I )\cap \{0\}\neq\phi$.
\end{proposition}

Similar to Thurston's condition (\ref{condition-Thurston}), conditions in Proposition (\ref{Proposition-topo-com-condition-ofGeXu}) also show that the combinatorial structure of the triangulation and the topology of surface, which have no relation with circle packing metrics, contains $\alpha$-curvature information surprisingly. Using Proposition \ref{Proposition-topo-com-condition-ofGeXu}, we can derive a combinatorial-topological-metric condition which contains Thurston's condition (\ref{condition-Thurston}) as a special case.

\begin{theorem}\label{Thm-combtopo-condition-CCCPmetric}
Suppose $(M, \mathcal{T}, \Phi)$ is a weighted triangulated surface with $\alpha\chi(M)\leq 0$.
Then there exists a constant $\alpha$-curvature metric iff. there exists a circle packing metric $r^*$ such that for any nonempty proper subset $I$ of vertices $V$,
\begin{equation}\label{combtopo-condition of GX}
2\pi\chi(M)\frac{\sum_{i\in I}r_i^{*\alpha}}{\|r^*\|^{\alpha}_{\alpha}}>-\sum_{(e,v)\in Lk(I)}(\pi-\Phi(e))+2\pi\chi(F_I).
\end{equation}
\end{theorem}
\textbf{Proof.}
By Lemma \ref{Lemma-admissible-curvature space}, for any circle packing metric $r$ and any proper subset $I\subset V$, the classical combinatorial Gauss curvature $K$ satisfies
\begin{equation}\label{K-satisfy}
\sum\limits_{i\in I}K_i(r)>-\sum_{(e,v)\in Lk(I)}(\pi-\Phi(e))+2\pi\chi(F_I).
\end{equation}
If there exists a constant $\alpha$-curvature metric $r^*$, substituting $K^*=K(r^*)=(K_1^*,\cdots,K_N^*)$ into (\ref{K-satisfy}), where $K_i^*=\frac{2\pi\chi(M)}{||r^*||^{\alpha}_{\alpha}}r_i^{*\alpha}$, we get (\ref{combtopo-condition of GX}). On the other hand, it's easy to see,

$\bullet$ when $\alpha>0$ and $\chi(M)<0$, (\ref{combtopo-condition of GX}) implies that $\mathscr{K}_{GB} \cap (\mathop{\cap}_{\phi\neq I \subsetneqq V} \mathscr{Y}_I )\cap\mathds{R}^N_{<0}\neq\phi$;

$\bullet$ when $\alpha<0$ and $\chi(M)>0$, (\ref{combtopo-condition of GX}) implies that $\mathscr{K}_{GB} \cap (\mathop{\cap}_{\phi\neq I \subsetneqq V} \mathscr{Y}_I )\cap\mathds{R}^N_{>0}\neq\phi$;

$\bullet$ when $\chi(M)=0$, (\ref{combtopo-condition of GX}) implies that $0\in\mathscr{K}_{GB} \cap (\mathop{\cap} _{\phi\neq I \subsetneqq V} \mathscr{Y}_I )$.

Proposition \ref{Proposition-topo-com-condition-ofGeXu} shows that conditions in above three cases all implies the existence of constant $\alpha$-curvature metrics. For $\alpha=0$ case, (\ref{combtopo-condition of GX}) is in fact Thurston's condition (\ref{condition-Thurston}). Thus we finish the proof.
 \qed\\
\begin{remark}
The proof of Proposition \ref{Proposition-topo-com-condition-ofGeXu} and hence Theorem \ref{Thm-combtopo-condition-CCCPmetric} deeply rely on deriving a discrete maximum principle for the flow (\ref{Def-new Ricci flow}). We want to know if there are more direct proofs without using discrete maximum principle for the flow (\ref{Def-new Ricci flow}).
\end{remark}

As to the uniqueness of constant $\alpha$-curvature metric, we restate Theorem 2.33 \cite{GX4} here for completeness.
\begin{theorem}\label{Thm-uniqueness-CCCPmetric}
Suppose $(M, \mathcal{T}, \Phi)$ is a weighted triangulated surface with $\alpha\chi(M)\leq 0$, then the constant $\alpha$-curvature metric is unique if it exists. Specificly, if $\alpha\chi(M)=0$, then there exists at most one constant $\alpha$-curvature metric up to scaling. If $\alpha\chi(M)<0$, then for any $c^*$, there exists at most one metric with $\alpha$-curvature $R_{\alpha,i}\equiv c^*$.
\end{theorem}

When $\alpha\chi(M)>0$, such as $\alpha=2$ and $M$ is a sphere, Example 2 in \cite{GX4} shows that the conclusions in Theorem \ref{Thm-convergence-negative Euler} are not true. For the tetrahedron triangulation of the sphere, if the initial metric $r(0)$ is close enough to $r^*=(1,\cdots,1)^T$, then the solution to flow (\ref{Def-Yamabe flow-2d}) converges to $r^*$ when $t\rightarrow -\infty$. However, the limit behavior of $r(t)$ depends on the selection of initial metric $r(0)$. Indeed, there exists $r(0)$ such that the solution $r(t)$ diverges to $\infty$ either $t$ tends to $+\infty$ or $-\infty$. Example 3 in \cite{GX4} shows that the constant $\alpha$-curvature metric is not unique. For the existence of constant $\alpha$-curvature, we have
\begin{corollary}\label{Thm-uniqueness-CCCPmetric}
Suppose $(M, \mathcal{T}, \Phi)$ is a weighted triangulated surface with $\alpha\chi(M)>0$. \\
(1)~There exists a constant $\alpha$-curvature metric.\\
(2)~There exists a circle packing metric $r^*$ such that for any nonempty proper subset $I$ of vertices $V$,
\begin{equation}\label{comb-topo-cond-for positive case}
2\pi\chi(M)\frac{\sum_{i\in I}r_i^{*\alpha}}{\|r^*\|^{\alpha}_{\alpha}}>-\sum_{(e,v)\in Lk(I)}(\pi-\Phi(e))+2\pi\chi(F_I).
\end{equation}
(3)~When $\alpha>0$ and $\chi(M)>0$, then $\mathscr{K}_{GB} \cap (\mathop{\cap}_{\phi\neq I \subsetneqq V} \mathscr{Y}_I )\cap\mathds{R}^N_{>0}\neq\phi$; When $\alpha<0$ and $\chi(M)<0$, then $\mathscr{K}_{GB} \cap (\mathop{\cap}_{\phi\neq I \subsetneqq V} \mathscr{Y}_I )\cap\mathds{R}^N_{<0}\neq\phi$.\\
Then (1) implies (2) which implies (3).
\end{corollary}
\textbf{Proof.}
For (1) $\Rightarrow$ (2), using formula (\ref{K-satisfy}). For (2) $\Rightarrow$ (3), it's obviously.
\qed\\

\begin{definition}\label{Def-modified-alpha-flow}
Suppose $(M, \mathcal{T}, \Phi)$ is a weighted triangulated surface with circle packing metric $r$, $\overline{R}\in C(V)$
is a function defined on $M$. The modified $\alpha$-flow with respect to $\overline{R}$ is defined to be
\begin{equation}\label{Equation-modified-alpha-flow}
\frac{du_i}{dt}=\overline{R}_ir_i^{\alpha}-K_i.
\end{equation}
\end{definition}

$\overline{R}$ is called admissible if there is a circle packing metric $\overline{r}$ with curvature $\overline{R}$. The modified $\alpha$-flow can be used to study prescribing curvature problem. On one hand, if the solution to the modified $\alpha$-flow (\ref{Equation-modified-alpha-flow}) converges, then $\overline{R}$ is admissible. On the other hand, we have
\begin{theorem}\label{Thm-prescribing curvature}
Suppose $(M, \mathcal{T}, \Phi)$ is a weighted triangulated surface and $\overline{R}\in C(V)$ is a function defined on $M$.
If $\alpha\overline{R}_i\leq 0$ for all $i$, but not identically zero, and $\overline{R}$ is admissible by a metric $\overline{r}$. Then $\overline{r}$ is the unique metric in $\mathds{R}^N_{>0}$ such that it's $\alpha$-curvature is $\overline{R}$. Moreover, the solution to the modified flow (\ref{Equation-modified-alpha-flow}) converges exponentially fast to $\overline{r}$.
\end{theorem}

\textbf{Proof.}
The first part is obviously, and $\overline{R}$ is admissible by metric $r(+\infty)$. For the second part, for given
function $\overline{R}\in C(V)$, we can introduce the following modified $\alpha$-potential
\begin{equation}\label{definition of modified Ricci potential}
\overline{F}(u)=\int_{u_0}^u\sum_{i=1}^N\left(K_i-\overline{R}_ir_i^{\alpha}\right)du_i.
\end{equation}
It is easy to check that the modified $\alpha$-potential $\overline{F}$ is well-defined. Furthermore, by direct
calculation, we have
\begin{equation*}
\begin{aligned}
Hess_u\overline{F}=L-\Sigma^{\frac{\alpha}{2}}
\left(
       \begin{array}{ccc}
       \alpha\overline{R}_1 &   &   \\
        & \ddots &   \\
        &   & \alpha\overline{R}_N \\
       \end{array}
\right)\Sigma^{\frac{\alpha}{2}}.
\end{aligned}
\end{equation*}
It is easy to check that, if $\alpha\overline{R}_i\leq 0$ for $i=1, \cdots, N$ and not identically zero, $Hess_u\overline{F}$ is
positive definite. By Lemma \ref{Lemma-injective}, $\nabla_u \overline{F}=(K_1-\overline{R}_1r_1^{\alpha},\cdots,K_N-\overline{R}_Nr_N^{\alpha})^T$
is an injective map from $u\in \mathds{R}^N$ to $\mathds{R}^N$. Hence $\overline{r}$ is the unique zero point of $\nabla_u \overline{F}$. This fact implies that $\overline{r}$ is the unique metric in $\mathds{R}^N_{>0}$ such that it's curvature is $\overline{R}$. By Lemma \ref{Lemma-proper}, we know that $\overline{F}$ is proper and $\lim\limits_{u\rightarrow\infty}\overline{F}(u)=+\infty$. Furthermore, $\frac{d}{dt}F(u(t))=-\sum_i(K_i-\overline{R}_ir_i^{\alpha})^2\leq 0$ implies that the solution of (\ref{Equation-modified-alpha-flow}) lies in a compact region. The following of the proof is the same as that of Theorem \ref{Thm-convergence-nonpositive Euler number}, so we omit it here. \qed\\

\begin{remark}\label{prescribing problem for R=0}
The second part of Theorem \ref{Thm-prescribing curvature} implies that $\alpha\chi(M)<0$. If $\alpha\overline{R}_i=0$ for all $i$, then the corresponding prescribing curvature problem is already solved in Theorem \ref{Thm-convergence-negative Euler}. In this case, the metric $\overline{r}$ is not unique. However, it's unique up to scaling. This is slightly different from Theorem \ref{Thm-prescribing curvature}.
\end{remark}

In the following of this section, we consider more general ``area element" $A_i$. It's interesting to define the ``$A$-curvature" as $R_i=K_i/A_i$, where $A_i>0$ is a function of circle packing metric $r\in \mathds{R}_{>0}^N$. We can consider the following discrete flow
\begin{equation}\label{Def-Aflow}
u'_i(t)=\frac{2\pi\chi(M)}{\sum A_i}A_i-K_i,
\end{equation}
which is called ``$A$-flow" for short. This generalized $A$-flow can be used to evolve a metric to a metric with constant $A$-curvature, i.e. a metric $r$ satisfying $K_i=sA_i$ for all $i\in V$, where $s=\frac{2\pi\chi(M)}{\sum A_i}$. It's easy to see the solution to $A$-flow always exists for all time $t\in (-\infty,+\infty)$. So this flow seems simpler than other flows such as $\dot{u}_i=s-R_i$, which is an $A$-generalization of flow (\ref{Def-new Ricci flow}). Furthermore, if the solution $r(t)$ to $A$-flow converges to $r(+\infty)\in\mathds{R}_{>0}^N$, then $r(+\infty)$ has constant $A$-curvature. It's very interesting that we can select $A_i$ as the real area instead of the area of disk packed at $i$, hence
\begin{equation}\label{Equation-sum-A=area-of-M}
\sum_{i=1}^NA_i=Area(M, \mathcal{T}, \Phi, r)
\end{equation}
is necessary, where $Area(M, \mathcal{T}, \Phi, r)$ is the total real area of a weighted triangulated surface $(M, \mathcal{T}, \Phi)$ with a fixed circle packing metric $r$. Then it's easy to see the following three selections of $A_i$ all satisfy (\ref{Equation-sum-A=area-of-M}).

\begin{example}
Select $A_i=\sum\limits_{\triangle ijk \in F}Area(\triangle ijk)/3$, where the sum is taken over all the triangles with $i$ as one of its vertices.
\end{example}

\begin{example}
Consider the dual structure determined by Thurston's circle patterns. For any face $\{ijk\}\in F$, denote $C_i,C_j,C_k$ as the closed disks centered at $i$, $j$ and $k$ so that their radii are $r_i$, $r_j$ and $r_k$. They both intersect with each other. Let $\mathcal{L}_i$, $\mathcal{L}_j$, $\mathcal{L}_k$ be the geodesic lines passing through the pairs of the intersection points of $\{C_k,C_j\}$, $\{C_k,C_i\}$, $\{C_i,C_j\}$. These three lines $\mathcal{L}_i$, $\mathcal{L}_j$, $\mathcal{L}_k$ must intersect in a common point $O_{ijk}$. Connect $O_{ijk}$ and $O_{ijl}$ whenever triangles $\{ijk\}$ and $\{ijl\}$ share a common edge $\{ij\}\in E$. Thus we get a dual graph. For more details see \cite{CL1,GX4,G3}. Select $A_i=Area(D_i)$, where $D_i$ is the dual $2$-cell of $i$.
\end{example}

\begin{example}
Inspired by \cite{MY}, we can select $A_i=Area(V_i)$, where $V_i$ is the Voronoi dual $2$-cell of $i$ in the Delaunay triangulation.
\end{example}

Under the assumption (\ref{Equation-sum-A=area-of-M}), the constant $A$-curvature is $\frac{2\pi\chi(M)}{Area(M, \mathcal{T}, \Phi, r)}$. If $(M, \mathcal{T}, \Phi, r)$ approximates a smooth Riemannian surface $(M, g)$, then $Area(M, \mathcal{T}, \Phi, r)$ approximates $Area(M, g)$. Hence the constant $A$-curvature $\frac{2\pi\chi(M)}{Area(M, \mathcal{T}, \Phi, r)}$ approximates the smooth average curvature $\frac{2\pi\chi(M)}{Area(M, g)}$=$\frac{\int_M Kdvol}{\int_M dvol}$. This fact inspires us to consider the following problem.
\begin{conjecture}
Fix a smooth Riemannian surface $(M, g)$. Suppose $(M_n, \mathcal{T}_n, \Phi_n, r_n)$ is a sequence of weighted triangulation of $M$ with initial circle packing metric $r_n$. $M_n$ is different with $(M, g)$ as metric space, although they are topologically equal. For each $n$, one can evolve $A$-flow (\ref{Def-Aflow}) and get a solution $r_n(t)$, $t\in [0,+\infty)$. Meanwhile, one can evolve $(M, g)$ by smooth Ricci flow and derive a solution $g(t)$, $t\in [0,+\infty)$. Assuming the initial Gromov-Hausdorff distance between $(M, g)$ and $(M_n, \mathcal{T}_n, \Phi_n, r_n)$ tends to zero, then $r_n(t)\rightarrow g(t)$ as $n\rightarrow+\infty$.
\end{conjecture}

\section{$\alpha$ order combinatorial flow in three dimension}\label{3-dimensional combinatorial Yamabe problem}
Suppose $M$ is a 3-dimensional compact manifold with a triangulation $\mathcal{T}=\{V,E,F,T\}$, where the symbols $V,E,F,T$ represent the sets of vertices, edges, faces and tetrahedrons respectively. A sphere packing metric is a map $r:V\rightarrow (0,+\infty)$ such that the length between
vertices $i$ and $j$ is $l_{ij}=r_{i}+r_{j}$ for each edge $\{i,j\}\in E$, and the lengths $l_{ij},l_{ik},l_{il},l_{jk},l_{jl},l_{kl}$ determines a Euclidean tetrahedron for each tetrahedron $\{i,j,k,l\}\in T$. Glickenstein pointed out \cite{G1} that a tetrahedron $\{i,j,k,l\}\in T$ generated by four positive radii $r_{i},r_{j},r_{k},r_{l}$ can be realized as a Euclidean tetrahedron if and only if
\begin{equation}\label{nondegeneracy condition}
Q_{ijkl}=\left(\frac{1}{r_{i}}+\frac{1}{r_{j}}+\frac{1}{r_{k}}+\frac{1}{r_{l}}\right)^2-
2\left(\frac{1}{r_{i}^2}+\frac{1}{r_{j}^2}+\frac{1}{r_{k}^2}+\frac{1}{r_{l}^2}\right)>0.
\end{equation}
Thus the space of admissible Euclidean sphere packing metrics is
$$\mathfrak{M}_{\mathcal{T}}=\left\{\;r\in\mathds{R}^N_{>0}\;\big|\;Q_{ijkl}>0, \;\forall \{i,j,k,l\}\in T\;\right\}.$$
Cooper and Rivin \cite{CR} called the tetrahedrons generated in this way conformal and proved that a tetrahedron is conformal if and only if there exists a unique sphere tangent to all of the edges of the tetrahedron. Moreover, the point of tangency with the edge $\{i,j\}$ is of distance $r_i$ to $v_i$. They further proved that $\mathfrak{M}_{\mathcal{T}}$ is a simply connected open subset of $\mathds{R}^N_{>0}$, but not convex.

For a triangulated 3-manifold $(M, \mathcal{T})$ with sphere packing metric $r$, there is also the notion of combinatorial scalar curvature.
Cooper and Rivin \cite{CR} defined combinatorial scalar curvature $K_{i}$ at a vertex $i$ as angle deficit of solid angles
\begin{equation}\label{Def-CR curvature}
K_{i}= 4\pi-\sum_{\{i,j,k,l\}\in T}\alpha_{ijkl},
\end{equation}
where $\alpha_{ijkl}$ is the solid angle at the vertex $i$ of the Euclidean tetrahedron $\{i,j,k,l\}\in T$ and the sum is taken over all tetrahedrons with $i$ as one of its vertices. For this curvature, Glickenstein \cite{G1} first defined a combinatorial Yamabe flow
\begin{equation}\label{Flow-Glickenstein}
\frac{dr_i}{dt}=-K_ir_i
\end{equation}
and give some very interesting and inspiring results.

Similar to the two dimensional case, Cooper and Rivin's definition of combinatorial scalar curvature $K_i$ is scaling invariant, which is not so satisfactory. The authors \cite{GX4} once defined a new combinatorial scalar curvatures as $R_i=K_i/r_i^2$ on 3-dimensional triangulated manifold $(M, \mathcal{T})$ with sphere packing metric $r$. Consider $r_i^2$ as the analogue of the smooth Riemannian metric. If $\widetilde{r}_i^2=c r_i^2$
for some positive constant $c$, we have $\widetilde{R}_i=c^{-1}R_i$. This is similar to the transformation of scalar curvature in smooth case under scaling. For this type of combinatorial scalar curvature, the authors defined a combinatorial Yamabe functional
\begin{equation}
Q(r)=\frac{\mathcal{S}}{V^{\frac{1}{3}}}=\frac{\sum_{i=1}^NK_ir_i}{(\sum_{i=1}^Nr_i^3)^{1/3}}, \ \ r\in\mathfrak{M}_{\mathcal{T}},
\end{equation}
and proposed to study the corresponding constant curvature problem which is called combinatorial Yamabe problem. For this, the authors defined a new discrete Yamabe flow
\begin{equation}
\frac{dg_i}{dt}=-R_ig_i,
\end{equation}
with normalization
\begin{equation}\label{normalized comb Yamabe flow}
\frac{dg_i}{dt}=(R_{av}-R_i)g_i,
\end{equation}
where $g_i=r_i^2$ and $R_{av}=\frac{\mathcal{S}}{\sum_{i=1}^Nr_i^3}$ is the average of the combinatorial scalar curvature.

Constant $R$-curvature metric means that $R_i\equiv$constant, which implies $K=R_{av}r^2$. The authors \cite{GX2} once defined the so called discrete quasi-Einstein metric satisfying $K=\lambda r$, which is similar to constant $R$-curvature metric. Motivated by these phenomena, we can generalize these properties to $\alpha$ order combinatorial scalar curvature (``$\alpha$-curvature" for short).

\begin{definition}\label{Def-alpha-curvature-3d}
For a triangulated 3-manifold $(M, \mathcal{T})$ with sphere packing metric $r$, the $\alpha$-curvature at the vertex $i$ is defined as
\begin{equation}\label{alpha-curvature 3d}
R_{\alpha,i}=\frac{K_i}{r_i^{\alpha}}
\end{equation}
for any $\alpha\in\mathds{R}$, where $K_i$ is given by (\ref{Def-CR curvature}).
\end{definition}

The study of smooth Einstein-Hilbert functional has a long history. For piecewise linear metric case, Regge \cite{Re} first give a discretization of this functional. For sphere packing metric case, the Einstein-Hilbert-Regge functional $\mathcal{S}=\sum_{i=1}^N K_i r_i$ was introduced by Cooper and Rivin in \cite{CR}.

\begin{definition}\label{Def-alpha-normalize-Regge-functional}
Suppose $(M, \mathcal{T})$ is a triangulated 3-manifold with a fixed triangulation $\mathcal{T}$. For any $\alpha\in \mathds{R}$, $\alpha\neq-1$, the $\alpha$ order combinatorial Yamabe functional (``$\alpha$-functional" for short) is defined as
\begin{equation}\label{Def-3d-Yamabe-functional}
Q_{\alpha}(r)=\frac{\mathcal{S}}{\|r\|_{\alpha+1}}=\frac{\sum_{i=1}^NK_ir_i}{\big(\sum_{i=1}^Nr_i^{\alpha+1}\big)^{\frac{1}{\alpha+1}}}, \ \ r\in\mathfrak{M}_{\mathcal{T}}.
\end{equation}
The $\alpha$ order combinatorial Yamabe invariant with respect to $\mathcal{T}$ is defined as
$$Y_{M,\mathcal{T}}=\inf_{r\in\mathfrak{M}_{\mathcal{T}}} Q_{\alpha}(r),$$
while the $\alpha$ order combinatorial Yamabe constant of $M$ is defined as $Y_{M}=\sup\limits_{\mathcal{T}}\inf\limits_{r\in\mathfrak{M}_{\mathcal{T}}} Q_{\alpha}(r).$
\end{definition}

When $\alpha\geq0$, then $|Q_{\alpha}(r)|\leq\|K\|_{1+\frac{1}{\alpha}}$. When $-1<\alpha<0$, then $|Q_{\alpha}(r)|\leq \sum_i|K_i|$. Hence the $\alpha$ order combinatorial Yamabe invariant $Y_{M,\mathcal{T}}$ is well defined when $\alpha>-1$. For $\alpha\geq0$ case, $Y_{M,\mathcal{T}}$ attains the minimum value $-\|K\|_{1+\frac{1}{\alpha}}$ at a metric $r^*$ if and only if $r^*$ is a constant $\alpha$-curvature metric with $s_{\alpha}^*\leq0$. As noted in \cite{GX4}, the admissible sphere packing metric space $\mathfrak{M}_{\mathcal{T}}$ for a given triangulated manifold $(M,\mathcal{T})$ may be considered as the combinatorial conformal class for $(M,\mathcal{T})$, which is an analogue of the conformal class $[g_0]$ of a Riemannian manifold $(M, g_0)$. Denote $s_{\alpha}=\mathcal{S}/\|r\|_{\alpha+1}^{\alpha+1}$. Then we have
\begin{equation}\label{gradient of Q-alpha}
\nabla_{r}Q_{\alpha}=\frac{K-s_{\alpha}r_i^{\alpha}}{\|r\|_{\alpha+1}}.
\end{equation}
Hence $r$ is a constant $\alpha$-curvature metric iff. it is
a critical point of $\alpha$-functional $Q_{\alpha}(r)$. We raise the following discrete $\alpha$-Yamabe problem
on 3-dimensional triangulated manifold.\\
~\\
\textbf{The Combinatorial $\alpha$-Yamabe Problem.}
Given a 3-dimensional manifold $M$ with triangulation $\mathcal{T}$,
find a sphere packing metric with constant combinatorial $\alpha$-curvature in the combinatorial
conformal class $\mathfrak{M}_{\mathcal{T}}$.\\

It's easy to see, if $r$ is a constant $\alpha$-metric with $K=\lambda r^{\alpha}$, then $\lambda=s_{\alpha}=\mathcal{S}/\|r\|_{\alpha+1}^{\alpha+1}$. To study the combinatorial $\alpha$-Yamabe problem, we introduce 3-dimensional $\alpha$-Yamabe flow.
\begin{definition}
Given a triangulated 3-manifold $(M, \mathcal{T})$ with sphere packing metric $r$. For any $\alpha\in\mathds{R}$, the $\alpha$ order combinatorial Yamabe flow (``$\alpha$-flow" for short) is
\begin{equation}\label{Def-alpha-Yamabe flow-3d}
\frac{dr_i}{dt}=s_{\alpha}r_i^{\alpha}-K_i.
\end{equation}
\end{definition}
\begin{remark}
The flow $\dot{r}=\lambda r-K$ ($\lambda=s_1$) introduced by the authors in \cite{GX2} is in fact the $\alpha$-flow defined above with $\alpha=1$. Hence $\alpha$-flow (\ref{Def-alpha-Yamabe flow-3d}) may be considered as a different normalization of the flow $\dot{r}=\lambda r-K$.
\end{remark}

Along the 3-dimensional $\alpha$-flow, $\|r(t)\|_2^2=\sum r_i^2(t)$ is invariant. Hence we always assume $r(0)\in\mathbb{S}^{N-1}$ in the following. It's easy to see that, if the solution of (\ref{Def-alpha-Yamabe flow-3d}) converges to a metric $r(\infty)$, then $r(\infty)$ is a metric with constant $\alpha$-curvature. It's interesting that almost all 3-dimensional results in \cite{GX2,GX4} are still true for $\alpha$-flow (\ref{Def-alpha-Yamabe flow-3d}). We just state some of them here.

\begin{theorem}\label{Thm-3d-compact-exist-const-alpha-metric}
If the solution of (\ref{Def-alpha-Yamabe flow-3d}) lies in a compact region in $\mathfrak{M}_{\mathcal{T}}\cap \mathbb{S}^{N-1}$, then there
exists at least one sphere packing metric with constant $\alpha$-curvature on $(M, \mathcal{T})$.\qed \\
\end{theorem}

\begin{definition}
Given a triangulated 3-manifold $(M, \mathcal{T})$. For any $\alpha\in \mathds{R}$, the $\alpha$ order combinatorial Laplacian (``$\alpha$-Laplacian" for short) $\Delta_{\alpha}:C(V)\rightarrow C(V)$ is defined as
\begin{equation}\label{Def-alpha-Laplacian}
\Delta_{\alpha} f_i=\frac{1}{r_i^{\alpha}}\sum_{j\sim i}(-\frac{\partial K_i}{\partial r_j}r_j)(f_j-f_i)
\end{equation}
for $f\in C(V)$.
\end{definition}

This definition of $\alpha$-Laplacian is a generalization of $\alpha=2$ case, which was carefully studied by the authors in \cite{GX4}. Similar to the two dimensional $\alpha$-Laplacian, the three dimensional $\alpha$-Laplacian (\ref{Def-alpha-Laplacian}) can also be written in a matrix form,
\begin{equation}
\Delta_{\alpha}=-\Sigma^{-\alpha}\Lambda \Sigma
\end{equation}
with $\Delta_{\alpha} f=-\Sigma^{-\alpha}\Lambda \Sigma f$ for each $f\in C(V)$, where $\Sigma=diag\big\{r_1,\cdots,r_N\big\}$ and
\begin{equation}
\Lambda=Hess_r\mathcal{S}=\frac{\partial(K_{1},\cdots,K_{N})}{\partial(r_{1},\cdots,r_{N})}.
\end{equation}
It was proved \cite{CR,Ri,G1,G2} that $\Lambda$ is positive semi-definite with rank $N-1$ and the kernel of $\Lambda$ is the linear space spanned by the vector $r$ (see Lemma \ref{Lemma-3d-Lambda matrix}).\\

Set $\Gamma_i(r)=s_{\alpha}r_i^{\alpha}-K_i$, $1\leq i\leq N$. Then the $\alpha$-flow (\ref{Def-alpha-Yamabe flow-3d}) can be written as $\dot{r}=\Gamma(r)$, which is an autonomy ODE system. Differentiate $\Gamma$, we get
\begin{equation*}
D\Gamma(r)=-\Lambda+\alpha s_{\alpha}\left(diag\left\{r_1^{\alpha-1},\cdots,r_N^{\alpha-1}\right\}-\frac{r^\alpha (r^\alpha)^T}{\|r\|_{\alpha+1}^{\alpha+1}}\right)-\frac{r^{\alpha}\left(K-s_{\alpha}r^{\alpha}\right)^T}{\|r\|_{\alpha+1}^{\alpha+1}}.
\end{equation*}
If $r^*\in\mathfrak{M}_{\mathcal{T}}$ is a sphere packing metric with constant $\alpha$-curvature, then
\begin{equation}\label{Diff-gamma at-r^*}
D\Gamma|_{r^*}=\left(-\Lambda+\alpha s_{\alpha}\left(diag\left\{r_1^{\alpha-1},\cdots,r_N^{\alpha-1}\right\}-\frac{r^\alpha (r^\alpha)^T}{\|r\|_{\alpha+1}^{\alpha+1}}\right)\right)_{r^*}.
\end{equation}

\begin{proposition}\label{Proposition-3d-semi-definite}
Given $(M^3, \mathcal{T})$, suppose $r^*$ is a constant $\alpha$-curvature metric. If the first positive eigenvalue of $-\Delta_{\alpha}$ at $r^*$ satisfies
\begin{equation}\label{3d-lamda1>alfa*s}
\lambda_1(-\Delta_{\alpha})>\alpha s_{\alpha}^*=\frac{\alpha\mathcal{S^*}}{\|r^*\|_{\alpha+1}^{\alpha+1}}
\end{equation}
then $-D\Gamma|_{r^*}$ is positive semi-definite with $rank$ $N-1$ and kernel $\{tr^*|t\in\mathds{R}\}$.
\end{proposition}
\textbf{Proof.}
Denote $\widetilde{\Lambda}=\Sigma^{\frac{1-\alpha}{2}}\Lambda \Sigma^{\frac{1-\alpha}{2}}$. Then
$$-\Delta_{\alpha}=\Sigma^{-\alpha}\Lambda \Sigma=\Sigma^{-\frac{1+\alpha}{2}}\widetilde{\Lambda}\Sigma^{-\frac{1+\alpha}{2}},$$
which implies that $$\lambda_1(-\Delta_{\alpha})=\lambda_1(\widetilde{\Lambda}).$$
Choose a matrix $P\in O(N)$ such that $P^T\widetilde{\Lambda}P=diag\{0,\lambda_1(\widetilde{\Lambda}),\cdots,\lambda_{N-1}(\widetilde{\Lambda})\}$.
Suppose $P=(e_0,e_1,\cdots,e_{N-1})$,
where $e_i$ is the $(i+1)$-column of $P$.
Then $\widetilde{\Lambda} e_0=0$ and $\widetilde{\Lambda} e_i=\lambda_i e_i,\,1\leq i\leq N-1$,
which implies $e_0=r^\frac{\alpha+1}{2}/\|r^\frac{\alpha+1}{2}\|$ and $r^\frac{\alpha+1}{2}\perp e_i,\,1\leq i\leq N-1$.
Hence $\left(I-\frac{r^{\frac{\alpha+1}{2}}(r^{\frac{\alpha+1}{2}})^T}{\|r\|^{\alpha+1}_{\alpha+1}}\right)e_0=0$ and $\left(I-\frac{r^{\frac{\alpha+1}{2}}(r^{\frac{\alpha+1}{2}})^T}{\|r\|^{\alpha+1}_{\alpha+1}}\right)e_i=e_i$, $1\leq i\leq N-1$.
Furthermore,
$$-D\Gamma|_{r^*}=\Sigma^{\frac{\alpha-1}{2}}P diag\left\{0,\lambda_1(\widetilde{\Lambda})-\alpha s_{\alpha}^*,\cdots,\lambda_{N-1}(\widetilde{\Lambda})-\alpha s_{\alpha}^*\right\}\Sigma^{\frac{\alpha-1}{2}}P^T.$$
Hence the conclusion is derived.\qed \\

\begin{theorem}\label{Thm-3d-isolat-const-alpha-metric}
The constant $\alpha$-curvature metrics satisfying $\lambda_1(-\Delta_{\alpha})>\alpha s_{\alpha}$ are isolated in $\mathfrak{M}_{\mathcal{T}}\cap \mathbb{S}^{N-1}$. Specifically, the constant $\alpha$-curvature metrics with $\alpha s_{\alpha}\leq0$ are isolated.
\end{theorem}
\textbf{Proof.}
Consider the map $\Gamma:\mathfrak{M}_{\mathcal{T}}\rightarrow \mathds{R}^{N}$, $r\mapsto \Gamma(r)$.
It is easy to see that the zero point of $\Gamma$ corresponds to the metric with constant $\alpha$-curvature. By (\ref{Diff-gamma at-r^*}) and Proposition \ref{Proposition-3d-semi-definite}, if the conditions in this theorem are satisfied, then $D\Gamma|_{r^*}$, the Jacobian of $\Gamma$ at the constant $\alpha$-curvature metric $r^*$, is negative semi-definite with $rank$ $N-1$ and kernel $\{tr^*|t\in\mathds{R}\}$. Notice that the kernel is the normal of $\mathbb{S}^{N-1}$. Restricted to $\mathbb{S}^{N-1}$, $D\Gamma$ is negative definite and then nondegenerate, which implies the conclusions. \qed\\

Proposition \ref{Proposition-3d-semi-definite} shows that constant $\alpha$-curvature metric $r^*$ with $\lambda_1(-\Delta_{\alpha}^*)>\alpha s_{\alpha}^*$ is a asymptotically stable point of the $\alpha$-flow. Hence we have

\begin{theorem}\label{Thm-3d-convergence of CYF under existence}
Given a triangulated manifold $(M^3, \mathcal{T})$. Suppose $r^*\in\mathbb{S}^{N-1}$ is a constant $\alpha$-curvature metric satisfying $\lambda_1(-\Delta_{\alpha}^*)>\alpha s_{\alpha}^*$, or more specifically, $r^*\in\mathbb{S}^{N-1}$ is a constant $\alpha$-curvature metric with $\alpha s_{\alpha}^*\leq0$. If $\|r(0)-r^*\|$ is small enough, then the solution of $\alpha$-flow (\ref{Def-alpha-Yamabe flow-3d}) exists for all time and converges to $r^*$.\qed\\
\end{theorem}

Let's look at what happens when $\alpha=0$. The $0$-curvature is just Cooper and Rivin's scalar curvature. By Theorem \ref{Thm-3d-isolat-const-alpha-metric} and Theorem \ref{Thm-3d-convergence of CYF under existence}, constant curvature metrics $r^*$ are always isolated in $\mathbb{S}^{N-1}$ and are always local attractors of the ODE system (\ref{Def-alpha-Yamabe flow-3d}).

\begin{remark}
$\alpha$-flow (\ref{Def-alpha-Yamabe flow-3d}) is not a negative gradient flow, since $\partial (s_{\alpha}r_i^{\alpha})/\partial r_j\neq \partial (s_{\alpha}r_j^{\alpha})/\partial r_i$. We could get similar results by considering the negative gradient flow of $\alpha$-functional (\ref{Def-3d-Yamabe-functional}), i.e.
\begin{equation}\label{3d-gradient-flow}
\dot{r}=-\nabla_rQ_{\alpha}=\frac{1}{\|r\|_{\alpha+1}}\left(s_{\alpha}r_i^{\alpha}-K\right).
\end{equation}
For this flow, Theorem \ref{Thm-3d-compact-exist-const-alpha-metric} and Theorem \ref{Thm-3d-convergence of CYF under existence} are still true.
\end{remark}

%$\frac{\partial s_{\alpha}r_i^{\alpha}}{r_j}\neq \frac{\partial s_{\alpha}r_j^{\alpha}}{r_i}$

\section{Some useful lemmas}\label{usful lemma}
For reader's convenience, we list some lemmas in this section which are used in the proof of our main results. Some of them are proved in detail while some of them we just give related references.
\begin{lemma}\label{Lemma-positive-definite}
Given a function $F\in C^2(\mathds{R}^N)$. Assuming $Hess(F)\geq0$, $rank(Hess(F))=N-1$, $\nabla F$ has at least one zero point and $F(u+t(1,\cdots,1)^T)=F(u)$ for any $u\in\mathds{R}^N$ and $t\in\mathds{R}$. Denote $\mathscr{U}=\{u\in\mathds{R}^N|\sum_iu_i=0\}$. Then the Hessian of $F\big|_{\mathscr{U}}$ (considered as a function of $N-1$ variables) is positive definite.
\end{lemma}
\textbf{Proof.} Chow and Luo \cite{CL1} stated similar results for $\alpha$-potential $F$ with $\alpha=0$. Since this lemma is very important, we give a direct and rigorous proof here for completeness.
Set $\gamma=(1,\cdots,1)^T/\sqrt{N}$. Choose $A\in O(N)$, such that $A\gamma=(0,\cdots,0,1)^T$, meanwhile, $A$ transforms $\mathscr{U}$ to $\{\zeta\in\mathds{R}^N|\zeta_N=0\}$. Define $g(\zeta_1,\cdots,\zeta_{N-1})\triangleq f(A^T(\zeta_1,\cdots,\zeta_{N-1},0)^T)$, we can finish the proof by showing that $Hess(g)$ is positive definite. Partition $A$ into two blocks, $A=\left[S^{N\times(N-1)}, \gamma\right]$. Write $L=Hess(F)$ for short, then $Hess(g)=S^TLS$. Notice that, $F(u+t\gamma)=F(u)$ implies $L\gamma=0$. Therefore we have
\begin{equation*}
A^TLA=
\begin{bmatrix}
S^T \\ \gamma^T
\end{bmatrix}
L\big[S,\gamma\big]=
\begin{bmatrix}
S^TLS & S^TL\gamma \\
\gamma^T LS & \gamma^T L\gamma
\end{bmatrix}
=\begin{bmatrix}
S^TLS& 0\,\,\\
0 & 0\,\,
\end{bmatrix},
\end{equation*}
which implies that $S^TLS$ is positive semi-definite. Furthermore, $rank(S^TLS)=N-1$. Hence $Hess(g)=S^TLS$ is positive definite, which implies that $F\big|_{\mathscr{U}}$ is strictly convex.\qed \\

\begin{lemma}\label{Lemma-proper}
Given a function $F\in C^2(\mathds{R}^N)$ with $Hess(F)>0$. Assuming $\nabla F$ has at least one zero point, then $\lim\limits_{\|x\|\rightarrow \infty} F(x)=+\infty$. Moreover, $F$ is proper.
\end{lemma}
\textbf{Proof.}
Without loss of generality, assume $\nabla F(0)=0$. First, we consider the case $N=1$. So we have $F: ~\mathds{R}\to \mathds{R}$ with $F''>0$ and $F'(0)=0$. Since $F''>0$ and $F'(0)=0$, then $F'(1)=a>0$ and $F'(-1)=-b<0$. Again, by $F''>0$, $F'(x)>F'(1)=a>0$ for $x>1$, and $F'(x)<F'(-1)=-b$. Integrating this gives
\begin{equation*}
\begin{aligned}
&F(x)\ge a(x-1)+F(1)~~~\mbox{ for $x>1$}\\
&F(x)\ge b(-1-x)+F(-1)~~~\mbox{ for $x<-1$}.
\end{aligned}
\end{equation*}
That means $F(x)\ge \min\{a,b\}|x|+C$, where $C=\min\{-a+F(1),-b+F(-1)\}$. Obviously, $\lim\limits_{|x|\to\infty}F(x)=+\infty$ and $F$ is proper.

For $N>1$, for any $\omega\in\mathbb{S}^{N-1}$, we consider the ray $\{t\omega\}_{0<t<\infty}\subset \mathds{R}^N$. Let
$f_\omega(t)=F(t\omega )$, then $f_\omega'(0)=0$ and $f_\omega''(t)=Hess_F(\omega,\omega)>0$, then as in the dimension $N=1$ case, there exists constant $a_\omega=f_\omega'(1)=\nabla F(\omega)\cdot\omega~>0$ and  such that
\begin{equation*}
\begin{aligned}
F(t\omega)=f_\omega(t)\ge a_\omega t-a_\omega+F(\omega),~~~\mbox{ for $t>1$}.
\end{aligned}
\end{equation*}
By the arguments above, we have proved $\omega\cdot \nabla F(\omega)>0$ for any $\omega\in \mathbb{S}^{N-1}$, by compactness of $\mathbb{S}^{N-1}$ and $F$ is $C^2$,
\begin{equation*}
\begin{aligned}
A:=\inf_{\omega\in \mathbb{S}^{N-1}}\omega\cdot \nabla F(\omega)>0.
\end{aligned}
\end{equation*}
Let $B:=\min\limits_{\omega\in \mathbb{S}^{N-1}}F(\omega)$, then we have
\begin{equation*}
\begin{aligned}
F(x)\ge A\|x\|-A+B,~~~\mbox{ for $\|x\|\ge 1$}.
\end{aligned}
\end{equation*}
It implies $\lim\limits_{\|x\|\to\infty}F(x)=+\infty$ and $F$ is proper.
\qed \\

\begin{lemma}(\cite{CL1})\label{Lemma-injective}
Suppose $\Omega\subset \mathds{R}^N$ is convex, the function $h:\Omega\rightarrow \mathds{R}$
is strictly convex, then the map $\nabla h:\Omega\rightarrow \mathds{R}^N$ is injective.
\end{lemma}

\begin{lemma}\label{Lemma-property of L}
(\cite{CL1}, Proposition 3.9.)
Suppose $(M, \mathcal{T}, \Phi)$ is a weighted triangulated surface. $r=(r_1,\cdots,r_N)^T$ is circle packing metric, while $K=(K_1,\cdots,K_N)^T$ is classical discrete Gauss curvature. Then $L=\frac{\partial(K_1, \cdots, K_N)}{\partial (u_1,\cdots,u_N)}$ is positive semi-definite with rank $N$-$1$ and kernel $\{t(1,\cdots,1)^T| t\in \mathds{R}\}$.
Moreover, $\frac{\partial K_i}{\partial u_i}>0$, $\frac{\partial K_i}{\partial u_j}<0$ for $i\sim j$ and $\frac{\partial K_i}{\partial u_j}=0$ for others.
\end{lemma}

\begin{lemma}\label{Lemma-admissible-curvature space}
(\cite{T1, MR, CL1})
Given a weighted triangulated surface $(M, \mathcal{T}, \Phi)$. Consider Thurston's circle packing metric $r$ and classical discrete Gauss curvature $K$. For any proper subset $I\subset V$, denote $\mathscr{Y}_I\triangleq\{x\in \mathds{R}^N |\sum_{i\in I}x_i >-\sum_{(e,v)\in Lk(I)}(\pi-\Phi(e))+2\pi\chi(F_I)\}$, $\mathscr{K}_{GB}\triangleq\{x\in \mathds{R}^N|\sum_{i=1}^Nx_i=2\pi\chi(X)\}$. Then the space of all admissible $K$-curvature is
$\{K=K(r)|r\in \mathds{R}^N_{>0}\}=\mathscr{K}_{GB} \cap (\mathop{\cap} _{\phi\neq I \subsetneqq V} \mathscr{Y}_I )$.
\end{lemma}

\begin{lemma}\label{Lemma-3d-Lambda matrix}
(\cite{CR, Ri, G1, G2})
Suppose $(M, \mathcal{T})$ is a triangulated 3-manifold with sphere packing metric $r$,
$\mathcal{S}=\sum K_ir_i$ is the Einstein-Hilbert-Regge functional. Then we have
\begin{equation}
\nabla_r\mathcal{S}=K.
\end{equation}
If we set
\begin{displaymath}
\Lambda=Hess_r\mathcal{S}=
\frac{\partial(K_{1},\cdots,K_{N})}{\partial(r_{1},\cdots,r_{N})}=
\left(
\begin{array}{ccccc}
 {\frac{\partial K_1}{\partial r_1}}& \cdot & \cdot & \cdot &  {\frac{\partial K_1}{\partial r_N}} \\
 \cdot & \cdot & \cdot & \cdot & \cdot \\
 \cdot & \cdot & \cdot & \cdot & \cdot \\
 \cdot & \cdot & \cdot & \cdot & \cdot \\
 {\frac{\partial K_N}{\partial r_1}}& \cdot & \cdot & \cdot &  {\frac{\partial K_N}{\partial r_N}}
\end{array}
\right),
\end{displaymath}
then $\Lambda$ is positive semi-definite with rank $N-1$ and
the kernel of $\Lambda$ is the linear space spanned by the vector $r$.
\end{lemma}

\textbf{Acknowledgements}\\[8pt]
The authors would like to thank Dr. Wenshuai Jiang, Liangming Shen for many helpful conversations. 
The first author would also like to give special thanks to Dr. Yurong Yu for her supports and encouragements during the work.
The research of the second author is partially supported by National Natural Science Foundation of China under grant no. 11301402 and 11301399.
He would also like to thank Professor Guofang Wang for the invitation to the Institute of Mathematics of the University of Freiburg 
and for his encouragement and many useful conversations during the work.

(Huabin Ge)
Department of Mathematics, Beijing Jiaotong University, Beijing 100044, P.R. China

E-mail: hbge@bjtu.edu.cn\\[2pt]

(Xu Xu) School of Mathematics and Statistics, Wuhan University, Wuhan 430072, P.R. China

E-mail: xuxu2@whu.edu.cn\\[2pt]

\end{document}